\documentclass[12pt]{amsart}
\usepackage{amsfonts,amsmath,amssymb,amscd}
\usepackage[dvips]{graphicx}
\usepackage[dvips]{psfrag}
\usepackage[dvips]{color}
\usepackage{xypic}
\usepackage[abs]{overpic}
\usepackage{caption}
\usepackage{pinlabel}
\usepackage[cmtip,all]{xy}
\usepackage[left=1.5in, top=1.5in, right=1.5in, bottom=1.5in]{geometry}

\newtheorem{theorem}{Theorem}[section]

\newtheorem{lemma}[theorem]{Lemma}
\newtheorem{cor}[theorem]{Corollary}

\theoremstyle{definition}

\makeatletter
\@namedef{subjclassname@2020}{\textup{2020} Mathematics Subject Classification}
\makeatother

\numberwithin{equation}{section}

\begin{document}

\title[Primitive disk complexes]{Note on primitive disk complexes}

\author[S. Cho]{Sangbum Cho}\thanks{
The first-named author is supported by the National Research Foundation of Korea(NRF) grant funded by the Korea government(MSIT) (NRF-202100000002302).}
\address{Department of Mathematics Education,
Hanyang University, Seoul 04763, Korea}
\email{scho@hanyang.ac.kr}

\author[J. H. Lee]{Jung Hoon Lee}
\address{Department of Mathematics and Institute of Pure and Applied Mathematics,
Jeonbuk National University, Jeonju 54896, Korea}
\email{junghoon@jbnu.ac.kr}



\begin{abstract}
Given a Heegaard splitting of the $3$-sphere, the primitive disk complex is defined to be the full subcomplex of the disk complex for one of the handlebodies of the splitting.
It is an open question whether the primitive disk complex is connected or not when the genus of the splitting is greater than three.
In this note, we prove that a quotient of the primitive disk complex, called the homotopy primitive disk complex, is connected.
\end{abstract}

\maketitle

\section{Introduction}

It is well known that every closed orientable $3$-manifold can be decomposed into two handlebodies of the same genus.
Such a decomposition is called a {\it Heegaard splitting} of the manifold, and the genus of the handlebodies is called the {\it genus} of the splitting.
In particular, the $3$-sphere admits a genus-$g$ Heegaard splitting for each $g \geq 0$.
We denote by $(V, W; \Sigma)$ a genus-$g$ Heegaard splitting of the $3$-sphere.
That is, $V$ and $W$ are genus-$g$ handlebodies whose union is the $3$-sphere, and $\Sigma = V \cap W = \partial V = \partial W$ is a genus-$g$ closed orientable surface, called the {\it Heegaard surface} of the splitting.

An essential disk $D$ in $V$ is called a {\it primitive disk} if there exists an essential disk $D'$ in $W$ such that $\partial D$ intersects $\partial D'$ transversely in a single point.
Such a disk $D'$ is called a {\it dual disk} of $D$.
The disk $D'$ is also primitive in $W$ with a dual disk $D$.
The {\it primitive disk complex} for the handlebody $V$, denoted by $\mathcal P(V)$, is defined as follows.
The vertices of $\mathcal P(V)$ are the isotopy classes of the primitive disks in $V$, and $k+1$ distinct vertices span a $k$-simplex if they have pairwise disjoint representatives.
Equivalently, the primitive disk complex is the full subcomplex of the disk complex for the handlebody $V$ spanned by the vertices whose representatives are primitive disks in $V$.
It is easy to see that the primitive disk complex $\mathcal P(V)$ for a genus-$g$ splitting of the $3$-sphere for $g \geq 2$ is a $(3g - 4)$-dimensional complex which is not locally finite.
For the genus-$0$ splitting, $\mathcal P(V)$ is empty and for the genus-$1$ splitting, $\mathcal P(V)$ consists only of a single vertex.

For the genus-$2$ splitting of the $3$-sphere, the combinatorial structure of $\mathcal P(V)$ has been well understood from \cite{C}.
It is a $2$-dimensional connected complex, even it is contractible.
For the genus-$3$ splitting of the $3$-sphere, Zupan \cite{Z} showed that the reducing sphere complex for the splitting is connected, which implies quickly that $\mathcal P(V)$ is connected either.
For the splittings of genus greater than $3$, it is still an open question whether $\mathcal P(V)$ is connected or not.

In this note, we define the {\it homotopy primitive disk complex}, denoted by $\mathcal{HP}(V)$, which can be considered as a quotient complex of $\mathcal P(V)$.
We say that two primitive disks $D$ and $E$ are {\it equivalent} if $\partial D$ and $\partial E$ represent the same element of the free group $\pi_1(W)$.
In other words, the boundary circles $\partial D$ and $\partial E$ in $\Sigma = \partial W$ are free homotopic in $W$.
Then vertices of $\mathcal{HP}(V)$ are defined to be the equivalence classes of primitive disks in $V$, and $k+1$ distinct vertices span a $k$-simplex if they have pairwise disjoint representatives.
The main result is the following theorem.

\begin{theorem}
\label{thm:main_theorem}
Let $D$ and $E$ be any two primitive disks in $V$.
Then there exists a sequence $D = E_1, E_2, \ldots, E_n$ of primitive disks in $V$ satisfying that
\begin{itemize}
\item $E_n$ is equivalent to $E$, and
\item for each $i \in \{1, 2, \ldots, n-1\}$, either $E_i = E_{i+1}$, or $E_i$ and $E_{i+1}$ are disjoint from each other and further $E_i$ and $E_{i+1}$ have dual disks $E'_i$ and $E'_{i+1}$ respectively such that $E'_i$ is disjoint from $E_{i+1} \cup E'_{i+1}$ and $E'_{i+1}$ is disjoint from $E_i \cup E'_i$.
\end{itemize}
\end{theorem}

The theorem implies the following immediately.

\begin{cor}
The homotopy primitive disk complex $\mathcal{HP}(V)$ is connected.
\end{cor}

Throughout the paper, $V$ and $W$ will denote two genus-$g$ handlebodies of the Heegaard splitting $(V, W; \Sigma)$ of the $3$-sphere for $g \geq 2$.
It is known that the Heegaard splitting of the $3$-sphere is unique up to isotopy for each genus \cite{W}.
Thus any orientation preserving homeomorphism of the splitting $(V, W; \Sigma)$ can be described as an isotopy from the identity map of the $3$-sphere to the homeomorphism.
We also remark that given any two primitive disks in $V$ there exists an orientation preserving homeomorphism of $(V, W; \Sigma)$ taking one to the other one.

A collection $\mathcal D = \{D_1, D_2, \ldots, D_g\}$ of pairwise disjoint, non-isotopic essential disks in $V$ is called a {\it primitive system} of $V$ if there exists a collection $\mathcal D' = \{D'_1, D'_2, \ldots, D'_g\}$ of pairwise disjoint, non-isotopic essential disks in $W$ such that $|\partial D_i \cap \partial D'_j| = \delta_{ij}$.
The collection $\mathcal D'$ is also a primitive system of $W$.
We say that $\mathcal D$ and $\mathcal D'$ are {\it dual systems} to each other, and that $\{\mathcal D, \mathcal D'\}$ simply a {\it dual pair}.
Note that $D_1 \cup D_2 \cup \cdots \cup D_g$ and $D'_1 \cup D'_2 \cup \cdots \cup D'_g$ cut $V$ and $W$ into $3$-balls respectively.

The fundamental group $\pi_1(W)$ of $W$ is the free group of rank $g$.
If any oriented simple closed curve $C$ in $\Sigma$ intersects $\partial D'_1 \cup \partial D'_2 \cup \cdots \cup \partial D'_g$ transversely, then $C$ represents an element of $\pi_1(W) = \langle x_1, x_2, \ldots, x_g \rangle$ in a natural way.
We first fix a base point inside $W$ and prepare an arc $a$ connecting the base point to a point on $C$ so that $a$ does not intersect each $D'_i$ for $i \in \{1, 2, \ldots, g\}$.
Then assigning symbols $x_1, x_2, \ldots, x_g$ to the oriented circles $\partial D'_1, \partial D'_2, \ldots, \partial D'_g$ respectively, we obtain a word $w$ in terms of $x_1^{\pm 1}, x_2^{\pm 1}, \ldots, x_g^{\pm 1}$ by reading off consecutive intersections of the loop $C \cup a$ with $\partial D'_1, \partial D'_2, \ldots, \partial D'_g$.
We say that such a word $w$ is {\it determined} by $C$ {\it with respect to} the system $\mathcal D'$.
Of course, the words $w$ determined by $C$ depend on the choice of the arc $a$, but they are all equivalent up to cyclic permutation.
In particular, the circles $\partial D_1, \partial D_2, \ldots, \partial D_g$ determine the generators $x_1, x_2, \ldots, x_g$ respectively.

\section{Proof of the main theorem}

It is well known that the automorphism group of $\pi_1(W) = \langle x_1, x_2, \ldots, x_g \rangle$ is generated by the {\it elementary Nielsen transformations} $\alpha, \beta, \gamma$, and $\delta$ described as follows.

\begin{itemize}
\item $\alpha(x_1) = x_1 x_2$, and $\alpha(x_j) = x_j$ for $j = 2, 3, \ldots, g$,
\item $\beta(x_1) = x^{-1}_1$ and $\beta(x_j) = x_j$ for $j = 2, 3, \ldots, g$,
\item $\gamma(x_1) = x_2$, $\gamma(x_2) = x_1$, and $\gamma(x_j) = x_j$ for $j = 3, 4, \ldots, g$, and
\item $\delta(x_j) = x_{j+1}$ for $j = 1, 2, \ldots, g-1$, and $\delta(x_g) = x_1$.
\end{itemize}

Any orientation preserving homeomorphism of the splitting $(V, W; \Sigma)$ induces an automorphism of $\pi_1(W)$ in a natural way.
Conversely, given an automorphism $\eta$ of $\pi_1(W)$, one can construct an orientation preserving homeomorphism $\phi_{\eta}$ of $(V, W; \Sigma)$ that induces $\eta$.
More precisely, we have the following lemma.

\begin{lemma}
\label{lem:Nielsen_transformations}
Let $\eta$ be any automorphism of $\pi_1(W)$, and let $D$ be a primitive disk in $V$.
Then there exists an orientation preserving homeomorphism $\phi_{\eta}$ of $(V, W; \Sigma)$ that induces $\eta$, and furthermore there exists a sequence $E_1, E_2, \ldots, E_n$ of primitive disks in $V$ satisfying that
\begin{itemize}
\item $E_1 = D$ and $E_n = \phi_{\eta}(D)$, and
\item for each $i \in \{1, 2, \ldots, n-1 \}$, either $E_i = E_{i+1}$, or $E_i$ and $E_{i+1}$ are disjoint from each other and they have dual disks $E'_i$ and $E'_{i+1}$ respectively such that $E'_i$ is disjoint from $E_{i+1} \cup E'_{i+1}$ and $E'_{i+1}$ is disjoint from $E_i \cup E'_i$.
\end{itemize}
\end{lemma}

\begin{proof}
We will describe an orientation preserving homeomorphism of $(V, W; \Sigma)$ as an isotopy from the identity to the homeomorphism.
Choose a primitive system $\mathcal D = \{D_1, D_2, \ldots, D_g\}$ of $V$ such that $D_1 = D$, together with a dual system $\mathcal D' = \{D'_1, D'_2, \ldots, D'_g\}$ of $W$, and then assign symbols $x_1, x_2, \ldots, x_g$ to the circles $\partial D'_1, \partial D'_2, \ldots, \partial D'_g$ respectively.
Fix a base point inside $W$ and choose pairwise disjoint arcs $a_i$, for $i \in \{1, 2, \ldots, g\}$, connecting the base point to a point in $\partial D_i$ respectively so that each $a_i$ is disjoint from $D'_1 \cup D'_2 \cup \cdots \cup D'_g$.
Then, with a choice of orientations of boundary circles of $D_i$ and of $D'_j$, for $i, j \in \{1, 2, \ldots, g\}$, the circles $\partial D_1, \partial D_2, \ldots, \partial D_g$ represent the generators $x_1, x_2, \ldots, x_g$ of $\pi_1(W) = \langle x_1, x_2, \ldots, x_g \rangle$ respectively.
See Figure \ref{handlebody}.

\bigskip
\bigskip

\begin{center}
\labellist
 \pinlabel {$D'_g$} [B] at -11 210
 \pinlabel {$D'_1$} [B] at 163 326
 \pinlabel {$D'_2$} [B] at 340 242
 \pinlabel {$D'_3$} [B] at 347 84
 \pinlabel {$\partial D_g$} [B] at 86 153
 \pinlabel {$\partial D_1$} [B] at 143 220
 \pinlabel {$\partial D_2$} [B] at 225 209
 \pinlabel {$\partial D_3$} [B] at 255 140
 \pinlabel {$a_g$} [B] at 126 176
 \pinlabel {$a_1$} [B] at 178 194
 \pinlabel {$a_2$} [B] at 220 162
 \pinlabel {$a_3$} [B] at 206 119
 \pinlabel {\Large$W$} [B] at 143 60
 \endlabellist
\includegraphics[width=8cm]{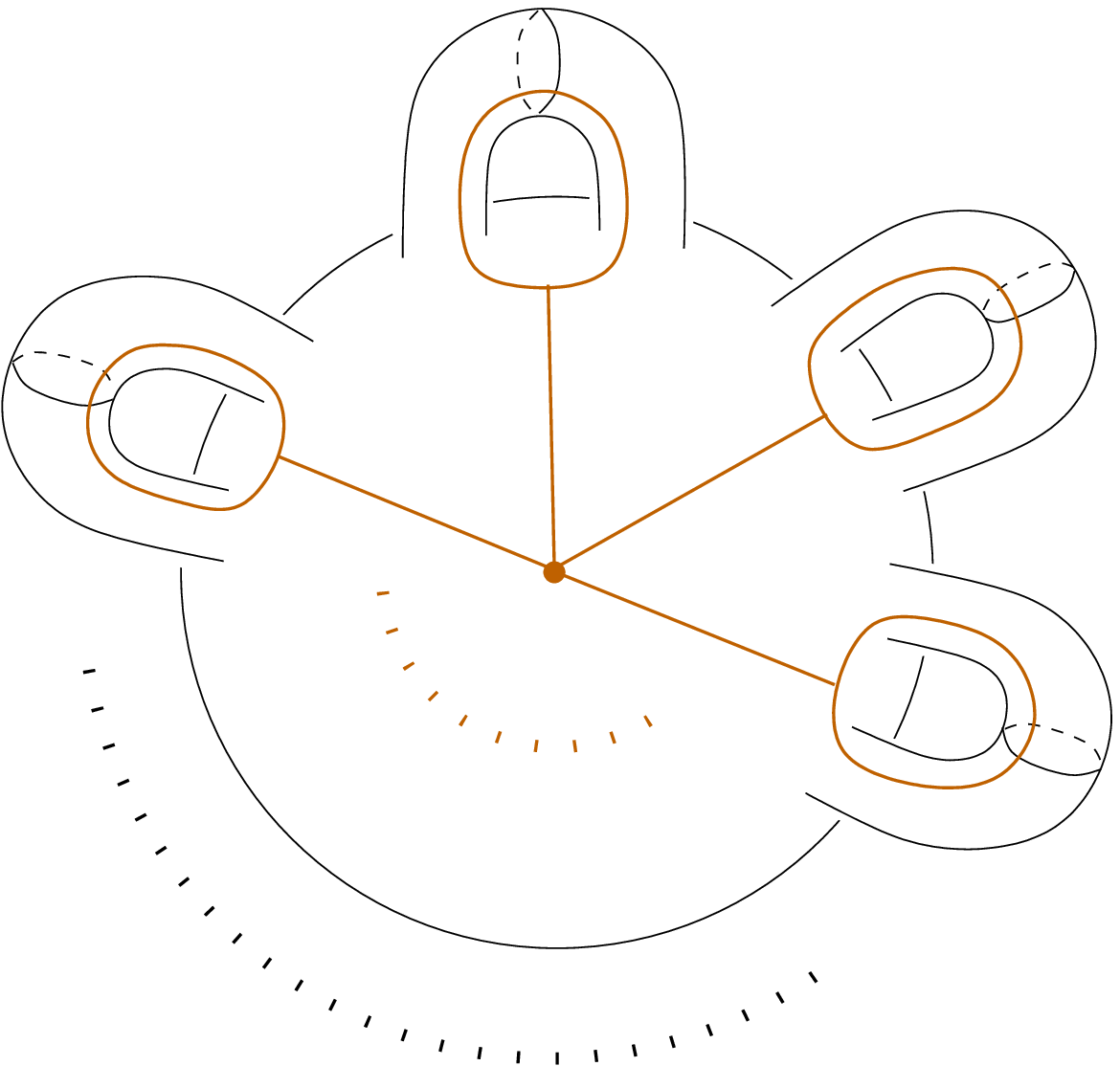}
\captionof{figure}{}
\label{handlebody}
\end{center}

\bigskip

First we construct the homeomorphisms $\phi_{\alpha}, \phi_{\beta}, \phi_{\gamma}$, and $\phi_{\delta}$ of the splitting $(V, W; \Sigma)$ that induce the four elementary Nielsen transformations $\alpha, \beta, \gamma$, and $\delta$ respectively.

For the transformation $\alpha$, we have the homeomorphism $\phi_{\alpha}$ of $(V, W; \Sigma)$ described by the isotopy in Figure \ref{homeomorphisms} (a).
We drag a foot of the first handle along the second handle to make the circle $\phi_{\alpha}(\partial D_1)$ determine the word $x_1 x_2$  with respect to $\mathcal D'$, while $\phi_{\alpha}(\partial D_j)$ determines $x_j$ for each $j \in \{2, 3, \ldots, g\}$.
After the isotopy for $\phi_{\alpha}$, all the disks in $\mathcal D \cup \mathcal D'$ except $D_1$ and $D'_2$ remain unchanged as subsets.

For the transformation $\beta$, the homeomorphism $\phi_{\beta}$ is described as a $\pi$-rotation of the first handle as in Figure \ref{homeomorphisms} (b).
The circle $\phi_{\beta}(\partial D_1)$ determines the word $x^{-1}_1$ with respect to $\mathcal D'$, while $\phi_{\beta}(\partial D_j)$ determines $x_j$ for $j \in \{2, 3, \ldots, g\}$.
After the isotopy for $\phi_{\beta}$, all the disks in $\mathcal D \cup \mathcal D'$ remain unchanged as subsets.

For the transformation $\gamma$, the homeomorphism $\phi_{\gamma}$ described by the isotopy in Figure \ref{homeomorphisms} (c) exchanges the first and the second handles so that $\phi_{\gamma}(\partial D_1)$ and $\phi_{\gamma}(\partial D_2)$ determine $x_2$ and $x_1$ respectively with respect to $\mathcal D'$, while $\phi_{\gamma}(\partial D_j)$ determines $x_j$ for $j \in \{3, 4, \ldots, g\}$.
The isotopy for $\phi_{\gamma}$ exchanges $D_1$ and $D'_1$ with $D_2$ and $D'_2$ respectively, but all other disks in $\mathcal D \cup \mathcal D'$ remain unchanged.

Finally, the homeomorphism $\phi_{\delta}$ described by the isotopy in Figure \ref{homeomorphisms} (d), which is a permutation of each handle to the next one, induces the transformation $\delta$.
For each $j \in \{1, 2, \ldots, g-1\}$, $\phi_{\delta}(\partial D_j)$ determines $x_{j+1}$, and $\phi_{\delta}(\partial D_g)$ determines $x_1$ with respect to $\mathcal D'$.
The isotopy for $\phi_{\delta}$ sends $D_j$ and $D'_j$ to $D_{j+1}$ and $D'_{j+1}$ respectively for $j \in \{1, 2, \ldots, g-1\}$, and sends $D_g$ and $D'_g$ to $D_1$ and $D'_1$ respectively.

\bigskip

\begin{center}
\labellist
 \pinlabel {(a)} [B] at 230 400
 \pinlabel {(b)} [B] at 230 270
 \pinlabel {(c)} [B] at 230 132
 \pinlabel {(d)} [B] at 230 0
 \pinlabel {$\phi_{\alpha}$} [B] at 229 442
 \pinlabel {$\phi_{\beta}$} [B] at 229 313
 \pinlabel {$\phi_{\gamma}$} [B] at 229 175
 \pinlabel {$\phi_{\delta}$} [B] at 229 36

 \pinlabel {$D'_1$} [B] at 59 480
 \pinlabel {$D'_1$} [B] at 59 347
 \pinlabel {$D'_1$} [B] at 59 210
 \pinlabel {$D'_1$} [B] at 59 71

 \pinlabel {$D'_2$} [B] at 152 480
 \pinlabel {$D'_2$} [B] at 152 347
 \pinlabel {$D'_2$} [B] at 152 210
 \pinlabel {$D'_2$} [B] at 152 71

 \pinlabel {$\partial D_1$} [B] at 59 397
 \pinlabel {$\partial D_1$} [B] at 59 265
 \pinlabel {$\partial D_1$} [B] at 79 130
 \pinlabel {$\partial D_1$} [B] at 84 -5

 \pinlabel {$\partial D_2$} [B] at 163 397
 \pinlabel {$\partial D_2$} [B] at 163 265
 \pinlabel {$\partial D_2$} [B] at 173 130
 \pinlabel {$\partial D_2$} [B] at 177 -5

 \pinlabel {$\phi_{\alpha}(D'_1)$} [B] at 308 483
 \pinlabel {$\phi_{\beta}(D'_1)$} [B] at 308 351
 \pinlabel {$\phi_{\gamma}(D'_1)$} [B] at 401 216
 \pinlabel {$\phi_{\delta}(D'_1)$} [B] at 401 74

 \pinlabel {$\phi_{\alpha}(D'_2)$} [B] at 355 469
 \pinlabel {$\phi_{\beta}(D'_2)$} [B] at 404 351
 \pinlabel {$\phi_{\gamma}(D'_2)$} [B] at 309 214
 \pinlabel {$\phi_{\delta}(D'_g)$} [B] at 309 75

 \pinlabel {$\phi_{\alpha}(\partial D_1)$} [B] at 279 397
 \pinlabel {$\phi_{\beta}(\partial D_1)$} [B] at 314 269
 \pinlabel {$\phi_{\gamma}(\partial D_1)$} [B] at 403 131
 \pinlabel {$\phi_{\delta}(\partial D_1)$} [B] at 430 -7

 \pinlabel {$\phi_{\alpha}(\partial D_2)$} [B] at 465 460
 \pinlabel {$\phi_{\beta}(\partial D_2)$} [B] at 407 269
 \pinlabel {$\phi_{\gamma}(\partial D_2)$} [B] at 314 131
 \pinlabel {$\phi_{\delta}(\partial D_g)$} [B] at 338 -7

 \endlabellist
\includegraphics[width=13cm]{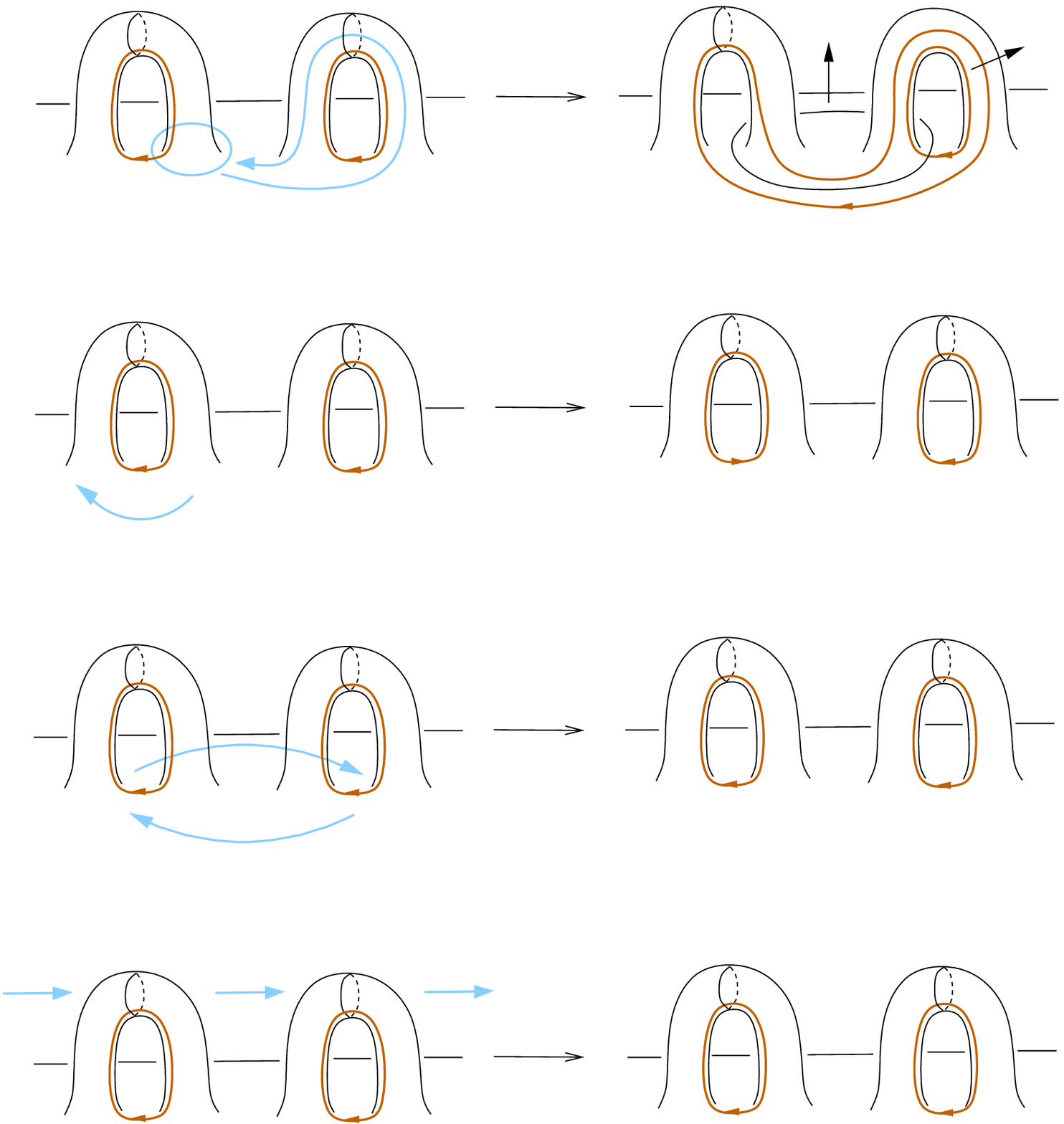}
\captionof{figure}{}
\label{homeomorphisms}
\end{center}

For each of the four homeomorphisms $\phi_{\alpha}$, $\phi_{\beta}$, $\phi_{\gamma}$, and $\phi_{\delta}$, we make the union of arcs $a_1 \cup a_2 \cup \cdots \cup a_g$ remain invariant after the isotopies.
Furthermore we observe that, if $\phi$ is one of $\phi_{\alpha}, \phi_{\beta}, \phi_{\gamma}$, or $\phi_{\delta}$, then either $D_j = \phi(D_j)$ or $D_j$ and $\phi(D_j)$ are disjoint from each other for each $j \in \{1, 2, \ldots, g\}$, and in this case we can easily find dual disks $E'$ and $E''$ of $D_j$ and $\phi(D_j)$ respectively, satisfying that $E'$ is disjoint from $\phi(D_j) \cup E''$ and $E''$ is disjoint from $D_j \cup E'$.
For example, for $D_1$ and $\phi_{\alpha} (D_1)$, we can take $E' = \phi_{\alpha} (D'_2)$ and $E'' = D'_2$.

Simply we say that the homeomorphisms $\phi_{\alpha}$, $\phi_{\beta}$, $\phi_{\gamma}$, and $\phi_{\delta}$ are {\it realizations} of $\alpha$, $\beta$, $\gamma$, and $\delta$ with respect to the dual pair $\{\mathcal D, \mathcal D'\}$.

Next, let $\eta$ be any automorphism of $\pi_1(W)$.
Then $\eta$ can be written as a finite product of elementary Nielsen transformations.
That is, $\eta = \eta_{n-1} \eta_{n-2} \cdots \eta_1$ for some $n \geq 2$ where $\eta_i$ is one of $\alpha$, $\beta$, $\gamma$, and $\delta$.
For $\eta_1$, we find a realization $\phi_1$, which is one of $\phi_{\alpha}$, $\phi_{\beta}$, $\phi_{\gamma}$, or $\phi_{\delta}$ with respect to the dual pair $\{\mathcal D, \mathcal D'\}$, described in the above.
For $\eta_2$, letting $\mathcal D_1 = \{\phi_1(D_1), \phi_1(D_2), \ldots, \phi_1(D_g)\}$ and $\mathcal D'_1 = \{\phi_1(D'_1), \phi_1(D'_2), \ldots, \phi_1(D'_g)\}$, we have a realization $\phi_2$, which is one of $\phi_{\alpha}$, $\phi_{\beta}$, $\phi_{\gamma}$, or $\phi_{\delta}$ with respect to the dual pair $\{\mathcal D_1, \mathcal D'_1\}$.
Continuing the process, we have a realization $\phi_i$ of $\eta_i$ for each $i \in \{2, \ldots, n-1\}$, with respect to the dual pair $\{\mathcal D_{i-1}, \mathcal D'_{i-1}\}$, where $\mathcal D_{i-1}$ and $\mathcal D'_{i-1}$ consist of disks $\phi_{i-1} \cdots \phi_2 \phi_1 (D_j)$ and $\phi_{i-1} \cdots \phi_2 \phi_1 (D'_j)$ respectively, for $j \in \{1, 2, \ldots, g\}$.
Then the product $\phi_{n-1} \phi_{n-2} \cdots \phi_1$ is the desired homeomorphism $\phi_{\eta}$ that induces $\eta$, and letting $D = E_1$, and $\phi_{i-1} \cdots \phi_2 \phi_1 (D) = E_i$ for $i \in \{2, 3, \ldots, n\}$ we have the desired sequence of primitive disks in $V$ satisfying the condition.
\end{proof}

\begin{proof}[Proof of Theorem $\ref{thm:main_theorem}$]
Choose any orientation preserving homeomorphism $\phi$ of $(V, W; \Sigma)$ taking $D$ to $E$.
Let $\eta$ be the automorphism of $\pi_1(W)$ induced by $\phi$.
By Lemma \ref{lem:Nielsen_transformations}, there exists another orientation preserving homeomorphism $\phi_{\eta}$ of $(V, W; \Sigma)$ that induces $\eta$, together with a sequence $E_1, E_2, \ldots, E_n$ of primitive disks in $V$ satisfying the condition in the theorem.
It is clear that $E = \phi(D)$ and $E_n = \phi_{\eta}(D)$ are equivalent to each other.
\end{proof}

\bibliographystyle{amsplain}

\end{document}